\documentclass[10pt]{article}
\usepackage{amsthm,amssymb,amsmath}

\newtheorem{definition}{Definition}[section]
\newtheorem{lemma}{Lemma}[section]
\newtheorem{theorem}{Theorem}[section]
\newtheorem{corollary}{Corollary}[section]
\newtheorem{problem}{Problem}[section]
\newtheorem{conjecture}{Conjecture}[section]

\begin{document}

\begin{center}

\large{\bf THE MULTISET DIMENSION OF GRAPHS}

\bigskip

\large{Rinovia Simanjuntak, Presli Siagian and Tom\'{a}\v{s} Vetr\'{i}k}

\end{center}

{\small
\noindent {\bf Abstract.}
We introduce a variation of metric dimension, called the multiset dimension. The representation multiset of a vertex $v$ with respect to $W$ (which is a subset of the vertex set of a graph $G$), $r_m (v|W)$, is defined as a multiset of distances between $v$ and the vertices in $W$. If $r_m (u |W) \neq r_m(v|W)$ for every pair of distinct vertices $u$ and $v$, then $W$ is called an m-resolving set of $G$. If $G$ has an m-resolving set, then the cardinality of a smallest m-resolving set is called the multiset dimension of $G$, denoted by $md(G)$. If $G$ does not contain an m-resolving set, we write $md(G) = \infty$.

In this paper we present basic results on the multiset dimension. We obtain some (sharp) bounds for multiset dimension of arbitrary graphs in term of its metric dimension, order, or diameter. We provide some necessary conditions for a graph to have finite multiset dimension, with an example of an infinite family of graphs where those necessary conditions are also sufficient. We also show that the multiset dimension of any graph other than a path is at least $3$ and finally we provide two families of graphs having the multiset dimension $3$.

\bigskip

\noindent {\bf Keywords:}
metric dimension, multiset dimension, distance.

\bigskip

\noindent {\bf Mathematics Subject Classification:} 05C35, 05C12.
}

\section{Introduction}

In July 2013, at Graph Master 2013, which was held at the University of Lleida in Spain, the first author presented a survey on metric dimension of graphs. The metric dimension was introduced separately by Slater \cite{S75} and Harary and Melter \cite{HM76}.

Let $G$ be a connected graph with the vertex set $V(G)$. The distance $d(u,v)$ between two vertices $u,v \in V(G)$ is the number of edges in a shortest path between them. A vertex $w$ resolves a pair of vertices $u,v$ if $d(u,w) \neq d(v,w)$. For an ordered set of $k$ vertices $W = \{ w_1, w_2, \dots , w_k \}$, the representation of distances of a vertex $v$ with respect to $W$ is the ordered $k$-tuple $$r(v|W) = ( d(v,w_1), d(v,w_2), \dots , d(v,w_z)).$$ A set of vertices $W \subset V(G)$ is a resolving set of $G$ if every two vertices of $G$ have distinct representations. A resolving set with minimum cardinality is called a metric basis and the number of vertices in a metric basis is called the metric dimension, denoted by $dim(G)$.

After Graph Master 2013, Charles Delorme, who was in the audience, send a suggestion to the first author to look at the multiset of distances, instead of looking at the vector of distances to some given set of vertices. Let us copy an excerpt of Charles' email.

\begin{quote}
\begin{small}
Instead of looking at the list of distances to some code in graph, we may look at the multiset of distances.

However, contrarily to the case of lists, where a sufficient large code allows the identification of all vertices, it may happen that no code provides identification. This is the case if the graph has too many symmetries.

For example, a complete graph with $n\ge 3$ vertices and a code with $m$ vertices (with $1\le m < n$) gives only two multisets, namely $\{1^n\}$ for vertices out of the code and $\{0,1^{n-1}\}$ for vertices in the code.

Some graphs however have a code such that the multiset of distances identifies the vertices. It is the case for cycles with $n\ge 6$ vertices.

Other example; Petersen graph. We recall that all graphs with less than $7$ vertices admit some non-trivial automorphism. Therefore, any code in Petersen graph with at most $6$ vertices contains two vertices with the same multiset of distances. On the other hand, if the code has $7$ or more vertices, some non-trivial automorphism of the whole graph preserves the complement of the code and the multisets in an orbit of this automorphism are the same.
\end{small}
\end{quote}

Charles fell ill at the end of 2013 and he was not in a good health condition until his passing away in 2015. This is one of the reasons why we did not pursue his idea further. Four years later, we reencountered Charles' email and felt that it is worthwhile to explore this new idea, since the notion is interesting and it has different properties than the original notion of metric dimension.

We start by formally defining the multiset dimension.

\begin{definition}
Let $G$ be a simple and connected graph with vertex set $V(G)$. Suppose that $W$ is a subset of $V(G)$ and $v$ is a vertex of $G$. The \emph{representation multiset of $v$ with respect to $W$}, $r_m (v|W)$, is defined as a multiset of distances between $v$ and the vertices in $W$. If $r_m (u |W) \neq r_m(v|W)$ for every pair of distinct vertices $u$ and $v$, then $W$ is called an \emph{m-resolving set} of $G$. If $G$ has an m-resolving set, then an m-resolving set having minimum cardinality is called a \emph{multiset basis} and its cardinality is called the \emph{multiset dimension} of $G$, denoted by $md(G)$; otherwise we say that $G$ has an infinite multiset dimension and we write $md(G) = \infty$.
\end{definition}

In this paper we present basic results on the multiset dimension. We obtain some (sharp) bounds for multiset dimension of arbitrary graphs in term of its metric dimension, order, or diameter (Chapter \ref{bounds}). We provide some necessary conditions for a graph to have finite multiset dimension, with an example of an infinite family of graphs where those necessary conditions are also sufficient (Chapter \ref{finite}). We also show that the multiset dimension of any graph other than a path is at least $3$ and finally we provide two families of graphs having the multiset dimension $3$ (Chapter \ref{3}).

\section{Some bounds} \label{bounds}

From the definitions it is clear that an m-resolving set of a graph $G$ is also a resolving set of $G$, which leads to the following.

\begin{lemma} \label{dim}
$md(G) \geq dim(G)$.
\end{lemma}

This bound is tight, since it is well-known that the metric dimension of a graph $G$ is one if and only if $G$ is a path and we obtain an equivalent result for the multiset dimension. Here we use $P_n$ to denote a path on $n$ vertices.

\begin{theorem}\label{path}
The multiset dimension of a graph $G$ is one if and only if $G$ is a path.
\end{theorem}
\begin{proof}
The set containing a pendant vertex of a path is an m-resolving set, thus $md(P_n)=1$. Now we show that if $md(G)=1$, then $G$ is a path. Let $W=\{ w \}$ be a multiset basis of a graph $G$. Then $d(u,w) \neq d(v,w)$ for any two vertices $u, v \in V(G)$, which means that there exists a vertex $x$ such that $d(x, w)=n-1$ (where $n$ is the order of $G$). This implies that the diameter of $G$ is $n-1$, hence $G$ is the path $P_n$.
\end{proof}

A more non-trivial example of a family of graphs with $md(G) = dim(G)$ is presented in Theorem \ref{md=diam}. To prove the theorem a few definitions and properties are in order. A \emph{major vertex of $G$} is a vertex of degree at least $3$. A pendant vertex $u$ of $G$ is a \emph{terminal vertex of a major vertex $v$ of $G$} if $d(u,v)<d(u,w)$ for every other major vertex $w$ of $G$. A major vertex $v$ of $G$ is an \emph{exterior major vertex of $G$} if it has at least one terminal vertex. Let $\sigma(G)$ denote the total number of terminal vertices of the major vertices of G and let $ex(G)$ denote the number of exterior major vertices of $G$. Then the following lower bound holds for general graphs.

\begin{lemma} \cite{CEJO00} \label{lbChartrand}
$dim(G)\geq \sigma(G) - ex(G)$.
\end{lemma}

By a \emph{branch of a vertex $u$} we mean a maximal subgraph with the vertex $u$ as an end point and by a \emph{path branch at $u$} we mean a branch of $u$ which is isomorphic to a path.
%Then it is quite obvious that the following lemmas hold.

%\begin{lemma} \label{branchpath}
%If $G$ contains two branch paths, then any m-resolving set of $G$ must contain at least a vertex from one of the branch paths.
%\end{lemma}

%\begin{lemma} \label{branch}
%Let $G$ be a graph rooted at a vertex $u$, where $u$ has two neighbours $u_1$ and $u_2$. Let $B_1$ and $B_2$ be the branches of $u_1$ and $u_2$, respectively. If $r_m(b_1|W\cap B_1)=r_m(b_2|W\cap B_2)$ for a vertex $b_1\in W\cap B_1$ and a vertex $b_2\in W\cap B_2$, then $W$ is not an m-resolving set.
%\end{lemma}

Let $H$ be a graph and $p$ a positive integer. We denote by $H^{(p)}$ the graph obtained from $H$ by subdividing each of its edges $p-1$ times.
It is easy to check that the metric dimension of $K_{1,n}^{(p)}$ is always $n-1$, but that is not the case for the multiset dimension.

\begin{theorem}\label{md=diam}
Let $K_{1,n}$ be a star on $n+1$ vertices.
For any positive integer $n$, $md(K_{1,n}^{(p)})=dim(K_{1,n}^{(p)}) = n-1$ if and only if $p\geq n-1$.
\end{theorem}
\begin{proof}
We denote by $v_c$ the center of the star.

Consider $p\geq n-1$. By Lemma \ref{lbChartrand}, $md(K_{1,n}^{(p)})\geq dim(K_{1,n}^{(p)})=n-1$. Now choose a set $W$ containing $n-1$ vertices, each from distinct subdivided edge and each has distinct distance from $v_c$. It is easy to see that $W$ is m-resolving and so $md(K_{1,n}^{(p)})=dim(K_{1,n}^{(p)})$.

%Now suppose that $p<n-1$ and consider an arbitrary set $W$ on $n-1$ vertices. If each vertex in $W$ is from distinct subdivided edge, then there exist two vertices of the same distance to the center of the star. In this case, by Lemma \ref{branch}, $W$ is not an m-resolving set. If there exist a subdivided edge containing more than one vertex of $W$, then there are at least two subdivided edges which are not containing vertices of $W$ and by Lemma \ref{branchpath}, $W$ is not an m-resolving set either. This completes the proof.

Let $p<n-1$ and consider an arbitrary set $W$ on $n-1$ vertices. If all vertices of $W$ are in different branch paths at $v_c$, then there must be (at least) two vertices of $W$ having same distance from $v_c$. Those vertices will then have the same representation. If all vertices of $W$ are not in different different branch paths at $v_c$, then there are at least $2$ different branch paths at $v_c$, say $P_1$ and $P_2$, containing no vertex of $W$ (we do not mind $v_c$ to be in $W$). Then the $i$-th vertex of $P_1$ and $P_2$ have the same representation for every $i = 1, 2, \dots, p$.
Therefore $W$ is not an m-resolving set.
\end{proof}

It is interesting to ask whether we could construct a graph where the gap between its metric and multiset dimensions could be as large as we want.
\begin{problem}
Let $c$ be a positive integer. Find a graph $G$ such that $md(G)=dim(G)+c$.
\end{problem}

Another lower bound is obtained by proving that no graph has multiset dimension $2$ and using the fact that only paths have multiset dimension 1 (Theorem \ref{path}). The sharpness of this bound will be discussed in Section \ref{3}.

\begin{lemma}\label{not2}
No graph has multiset dimension $2$.
\end{lemma}
\begin{proof}
Assume that $md(G) = 2$ for some graph $G$. Let $W = \{w_1, w_2\}$ be an m-resolving set of $G$. Then $r_m(w_1|W) = \{0, d(w_1, w_2)\} = r_m(w_2|W)$, a contradiction.
\end{proof}

\begin{theorem} \label{bound1}
Let $G$ be a graph other than a path. Then $md(G) \ge 3$.
\end{theorem}

Another bound relates the multiset dimension of a graph with its diameter. Theorem \ref{nd} gives a better lower bound than the one presented in Theorem \ref{bound1}. For positive integers $n$ and $d$, we define $f(n,d)$ to be the least positive integer $k$ for which $\frac{(k+d-1)!}{k!(d-1)!} + k \ge n$.

\begin{theorem}\label{nd}
If $G$ is a graph of order $n \ge 3$ and diameter $d$, then $md(G) \ge f(n,d)$.
\end{theorem}
\begin{proof}
Let $W$ be a multiset basis of $G$ having $k$ vertices. If $x$ is a vertex not in $W$, then $r_m(x|W)=\{1^{m_1}, 2^{m_2},\dots, d^{m_d}\}$, where $m_1 + m_2 + \dots + m_d = k$ and $0 \le m_i \le k$ for each $i = 1, 2, \dots , k$. Then there are $C(k+d-1,d-1)$ different possibilities for representation of $x$. Since we have $n-k$ vertices not in $W$, $\frac{(k+d-1)!}{k!(d-1)!}$ must be at least $n-k$. Hence $\frac{(k+d-1)!}{k!(d-1)!} + k \ge n$.
\end{proof}

\begin{problem}
Is the bound in Theorem \ref{nd} sharp?
\end{problem}

So far we only obtain lower bounds for the multiset dimension of an arbitrary graph. If a graph $G$ on $n$ vertices has finite multiset dimension, it is obvious that $md(G)\leq n$. However we suspect that this bound is not sharp and so we suggest the following conjecture.
\begin{conjecture} \label{n-1}
If $G$ is a graph on $n$ vertices having finite multiset dimension, then $md(G)\leq n-1$.
\end{conjecture}

Combining Theorem \ref{bound1} and Conjecture \ref{n-1}, we ask the following.
\begin{problem} \label{allc}
For any integer $3\leq c \leq n-1$, find a graph $G$ with $md(G)=c$.
\end{problem}

It is also interesting to bound the multiset dimension by using other graph parameters. Here we propose to study two particular parameters.
\begin{problem}
Let $G$ be a graph on $m$ edges. If $md(G)$ is finite, find lower and upper bounds on $md(G)$ with respect to $m$.
\end{problem}

\begin{problem}
Let $G$ be a graph with minimum degree $\delta$. If $md(G)$ is finite, find bounds on $md(G)$ with respect to $\delta$.
\end{problem}

\section{Finiteness of multiset dimension} \label{finite}

In this section we give two necessary conditions for a graph to have finite multiset dimension.

\begin{lemma}\label{dist2}
Let $G$ be a graph and let $W'$ be a set of vertices, where $|W'| \ge 2$.
If every pair of vertices in $W'$ is of distance at most $2$, then $W'$ is not an m-resolving set of $G$.
\end{lemma}
\begin{proof}
We prove the result by contradiction. Assume that every pair of vertices in $W$ is of distance at most $2$ and $W$ is an m-resolving set of $G$. We denote the vertices in $W$ by $w_1, w_2, \dots , w_p$, where $p \ge 2$. For $i = 1, 2, \dots p$, we have $r_m(w_i | W) = \{0, 1^{m_1}, 2^{m_2}\}$, where $m_1 + m_2 = p-1$. Since we have $p$ vertices in $W$ and they have different representations, their representations must be
$$\{0, 1^{p-1}\}, \{0, 1^{p-2}, 2\}, \dots , \{0, 1, 2^{p-2}\}, \{0, 2^{p-1}\}.$$
Without loss of generality we can assume that the vertex having the representation $\{0, 1^{p-1}\}$ is $w_1$ and the vertex having the representation $\{0, 2^{p-1}\}$ is $w_p$. Since $r_m(w_1 | W) = \{0, 1^{p-1}\}$, it follows that $d(w_1, w_p) = 1$, a contradiction to $r_m(w_p | W) = \{0, 2^{p-1}\}$. Hence, $W$ is not an m-resolving set of $G$.
\end{proof}

In a graph of diameter at most $2$, the distance between any two vertices is at most $2$, thus Theorem \ref{diam2} is a corollary of Lemma \ref{dist2}.

\begin{theorem}\label{diam2}
If $G$ is a non-path graph of diameter at most $2$, then $md(G) = \infty$.
\end{theorem}

This means that cycles with at most $5$ vertices, complete graphs, stars, the Petersen graph and strongly regular graphs have infinite multiset dimension. Since it is well known that ``almost all graphs have diameter 2" \cite{MM66}, we could also state that the following corollary.

\begin{corollary}\label{almostall}
Almost all graphs have infinite multiset dimension.
\end{corollary}

The last necessary condition needs the definition of twins as follow. Two vertices $u$ and $v$ are said to be \emph{twins} if $N(u)\setminus\{v\}=N(v)\setminus\{u\}$. If $u$ and $v$ are twins in $G$, then $d(u,x)=d(v,x)$ for any other vertex $x$ in $V(G)\setminus\{u,v\}$. We then define a relation $\sim$ where $u \sim v$ if and only if $u=v$ or $u$ and $v$ are twins. It is quite obvious that $\sim$ is an equivalence relation on $V(G)$ (see \cite{HMPSW10}). We shall denote by $v^*$ the equivalence class containing the vertex $v$. The following lemma shows the relation between the cardinality of an equivalence class with the multiset dimension.

\begin{lemma} \label{twins}
If $G$ contains a vertex $v$ with $|v^*|\geq 3$, then $md(G) = \infty$.
\end{lemma}
\begin{proof}
Let $W$ be any m-resolving set of $G$ and $v_1, v_2, v_3$ be three vertices in $[v]$. Then either at least two of $v_1, v_2, v_3$ (say $v_1$ and $v_2$) are in $W$, or at least two of them (say $v_1$ and $v_2$) are not in $W$. In both cases these vertices cannot be resolved, because $d(v_1, x) = d(v_2, x)$ for any other vertex $x \in V(G)\setminus\{v_1,v_2\}$.
\end{proof}

It is clear that pendant vertices adjacent to a particular vertex are members of the same equivalence class. This leads to the following corollary.
\begin{corollary} \label{3pendants}
If $G$ contains a vertex which is adjacent to (at least) three pendant vertices, then $md(G) = \infty$.
\end{corollary}

From Lemma \ref{twins}, in order to ensure that a graph $G$ has a finite multiset dimension, the cardinality of each equivalence class $v^*$ is either 1 or 2. In particular, when $|v^*|=2$, then exactly one of the two vertices in $v^*$ must be included in any m-resolving set of $G$, otherwise the two vertices have the same distances to any other vertices in $G$.

\begin{lemma} \label{choose1}
Let $G$ be a graph with finite multiset dimension. If $v$ is a vertex in $G$ with $|v^*|= 2$, then any m-resolving set of $G$ must contain exactly one vertex in $v^*$.
\end{lemma}

%\section{Multiset dimension of trees}

The necessary conditions in Theorem \ref{diam2} and Lemma \ref{twins} are, unfortunately, not sufficient for a graph to have finite multiset dimension. We shall give an example of a tree of diameter 4, having no vertex $v$ with $|v^*|\geq 3$, but has infinite multiset dimension. Let $T$ be the rooted tree of height $2$ with the root vertex having $3$ neighbours, each of them adjacent to $2$ pendant vertices. Let $W$ be any m-resolving set of $T$. By Lemma \ref{choose1}, exactly one pendant vertex from each pair must be in $W$. Now, consider the $3$ vertices in the first level. Either at least two of them are in $W$ or at least two of them are not in $W$. In both cases, the two vertices have the same representations with respect to $W$, regardless what the members of $W$ are.

However we shall present a family of trees where the conditions in Theorem \ref{diam2} and Lemma \ref{3pendants} are both necesarry and sufficient for the tree to have finite multiset dimension.

\begin{theorem}
The multiset dimension of a complete $k$-ary tree is finite if and only if $k=1$ or $2$. Moreover, if $T$ is a complete binary tree of height $h$, then $md(T)=2^h-1$.
\end{theorem}
\begin{proof}
Let $T$ be a complete $k$-ary tree of height $h \ge 1$. If $k \ge 3$, then by Lemma \ref{3pendants}, $md(T) = \infty$.
If $k = 1$, then $T$ is a path and $md(T) = 1$ (by Theorem \ref{path}).

Let $k=2$. If $h = 1$, then $T$ is a path having two edges and $md(T) = 1$. So, let us study binary trees of height $h \ge 2$.
Let $W$ be any m-resolving set of $T$.
Consider the last level $h$ of $T$ containing $2^{h-1}$ pairs of pendant vertices of distance $2$. Note that exactly one pendant vertex of each pair must be in $W$ (otherwise their representations would be equal).
Now consider level $h-1$ of $T$ having $2^{h-2}$ pairs of vertices of distance $2$.
Vertices of each pair have the same representations with respect to the vertices in $W$ which are in level $h$ and they have the same distance to any other vertex of $T$. So the pair cannot be resolved by other vertices, which means that exactly one of the vertices of each pair is in $W$.
We can repeat similar arguments for next levels (levels $h-2$, $h-3$, \dots , $1$) to obtain
$2^{h-1}+2^{h-2}+ \dots +1 = 2^h-1$ vertices that must be in $W$, thus $md(T) \geq 2^h-1$.

In each level $i$ of $T$, where $1\leq i \leq h$, there are exactly $2^{i-1}$ pairs of vertices of distance $2$. Let $W$ contains exactly one vertex of each such pair. So $|W|=\sum_{i=1}^{h} 2^{i-1} = 2^h-1$.
We prove that $W$ is an m-resolving set. Let us show that any two vertices $u, v$ of $T$ are resolved by $W$.
We consider two cases.
\begin{itemize}
\item[(1)] $u$ and $v$ are in different levels, say $i$ and $j$, where $0 \le i < j \le h$:

The distance between $v$ and $2^{h-2}$ vertices of $W$ which are in level $h$ is $j+h$.
On the other hand, there is no vertex in $W$ of distance $j+h$ from $u$. Thus
$u$ and $v$ have different representations.

\item[(2)] $u$ and $v$ are in the same level $i$, where $1 \le i \le h$:

If exactly one of them is in $W$, clearly they have different representations.
If none of them is in $W$ or both of them are in $W$, then let us denote by $x$ the central vertex of the path connecting $u$ and $v$. This path has even number of edges, say $2s$, and then $x$ is in level $i - s$. We know that $x$ is adjacent to two vertices, say $x_1$, $x_2$, in level $i - s + 1$. Both, $x_1$ and $x_2$, belong to the path and exactly one of them, say $x_1$, is in $W$. Clearly $d(u, x_1) \neq d(v, x_1)$ and it can be checked that $u$ and $v$ have the same representations with respect to $W \setminus \{ x_1 \}$. Hence $W$ is an m-resolving set and
$md(T) \leq 2^h-1$. The proof is complete.
\end{itemize}
\end{proof}

We suspect that characterizing all graphs with finite multiset dimension will be difficult, however we would like to ask similar question for trees in the following.
\begin{problem}
Characterize all trees having finite multiset dimension. Give exact values of the multiset dimension of trees if it is finite.
\end{problem}

\section{Graphs with multiset dimension $3$} \label{3}

From Theorem \ref{bound1} we know that the multiset dimension of any graph other than a path is at least $3$.
Here we present two families of graphs showing that the lower bound is sharp.
The proof for the first result was given by Charles Delorme.

\begin{theorem}\cite{Delorme}
Let $n\ge 6$. The multiset dimension of the cycle $C_n$ is $3$.
\end{theorem}
\begin{proof}
The description of the representations of vertices depends on the parity of $n$; in both cases, we check that the set $W=\{v_0,v_1,v_3\}$ with a usual labelling of the cycle is convenient.

If $n=2t$ with $t\ge 3$, the representations of vertices are the following.
{\small
\[\begin{array}{cccc}
v_0&v_1&v_2&v_3\\
\{0,1,3\}&\{0,1,2\}&\{1,1,2\}&\{0,2,3\}\\
v_i \ (3 < i < t)&v_t&v_{t+1}&v_{t+2}\\
\{i-3,i-1,i\}&\{t-3,t-1,t\}&\{t-2,t-1,t\}&\{t-2,t-1,t-1\}\\
v_{t+3}&v_{i+t} \ (3 < i < t)\\
\{t-3,t-2,t\}&\{t-i,t-i+1,t-i+3\}\\
\end{array}\]}

If $n=2t+1$ with $t\ge 3$, the representations of vertices are given by:
{\small\[\begin{array}{cccc}
v_0&v_1&v_2&v_3\\
\{0,1,3\}&\{0,1,2\}&\{1,1,2\}&\{0,2,3\}\\
v_i \ (3 < i < t)&v_t&v_{t+1}&v_{t+2}\\
\{i-3,i-1,i\}&\{t-3,t-1,t\}&\{t-2,t,t\}&\{t-1,t-1,t\}\\
v_{t+3}&v_{t+4}&v_{i+t+1} \ (3 < i < t)\\
\{t-2,t-1,t\}&\{t-3,t-2,t\}&\{t-i,t-i+1,t-i+3\}\\
\end{array}\]}
\end{proof}

In the following, we denote by $G \Box H$ the Cartesian product of the graphs $G$ and $H$.
\begin{theorem}
Let $m \ge 3$ and $n \ge 2$. The multiset dimension of the grid graph $P_m \Box P_n$ is $3$.
\end{theorem}
\begin{proof}
We can write $V(P_m \Box P_n) = \{ v_{i,j} \ | \ i = 1, 2, \dots, m, \ j = 1, 2, \dots, n \}$.
Then $E(P_m \Box P_n) = \{ v_{i,j}v_{i,j+1} \ | \ i = 1, 2, \dots, m, \ j = 1, 2, \dots, n-1 \} \cup
\{ v_{i,j}v_{i+1,j} \ | \ i = 1, 2, \dots, m-1, \ j = 1, 2, \dots, n \}$.

Let us show that $W = \{v_{1,1},v_{1,2},v_{3,1} \}$ is a m-resolving set of
the graph $P_m \Box P_n$. We present representations of vertices with respect to $W$ as follows
\begin{eqnarray*}
r_m (v_{1,1} | W) & = & \{ 0, 1, 2 \},\\
r_m (v_{2,1} | W) & = & \{ 1, 1, 2 \},\\
r_m (v_{i,1} | W) & = & \{ i-3, i-1, i \} \ \text{for} \ 3 \le i \le m,\\
r_m (v_{1,j} | W) & = & \{ j-2, j-1, j+1 \} \ \text{for} \ 2 \le j \le n,\\
r_m (v_{2,j} | W) & = & \{ j-1, j, j \} \ \text{for} \ 2 \le j \le n,\\
r_m (v_{i,j} | W) & = & \{ i+j-4, i+j-3, i+j-2 \} \ \text{for} \ 3 \le i \le m, \ 2 \le j \le n.
\end{eqnarray*}
Since no two vertices have the same representations, $W$ is an m-resolving set
and hence $md(P_m \Box P_n) = 3$.
\end{proof}

We conclude this section by asking the following.

\begin{problem}
Characterize all graphs with multiset dimension 3.
\end{problem}

\section{Multiset dimension for Cayley graphs}

Since the multiset dimension is a new invariant, there is a very large space for research in this area. We have stated a few interesting problems and conjecture in the previous sections and we would like to add an additional problem as stated below.

\begin{problem}
Study the multiset dimension for Cayley graphs of cyclic, Abelian, and non-Abelian groups.
\end{problem}

\section*{Acknowledgments}

This research was initiated as a problem for the second author's thesis and continued when the third author visited the Institut Teknologi Bandung with funding from the World Class University Grant 2017. The first author was supported by Penelitian Dasar Unggulan Perguruan Tinggi 2017-2019, funded by Indonesian Ministry of Research, Technology and Higher Education.

\bigskip

\noindent Rinovia Simanjuntak

\noindent rino@math.itb.ac.id

\bigskip

\noindent Institut Teknologi Bandung

\noindent Combinatorial Mathematics Research Group

\noindent Jl Ganesa 10, Bandung 40132, Indonesia

\bigskip

\noindent Presli Siagian

\noindent presdi453@gmail.com

\bigskip

\noindent Institut Teknologi Bandung

\noindent Master Program in Mathematics

\noindent Jl Ganesa 10, Bandung 40132, Indonesia

\bigskip

\noindent Tom\'{a}\v{s} Vetr\'{i}k

\noindent vetrikt@ufs.ac.za

\bigskip

\noindent University of the Free State

\noindent Department of Mathematics and Applied Mathematics

\noindent P. O. Box 339, 9300 Bloemfontein, South Africa

\end{document}